\documentclass{amsart}
\usepackage{graphicx}
\usepackage{amscd}
\usepackage{color}
\usepackage{epstopdf}
\usepackage{pinlabel}
\newtheorem{theorem}{Theorem}[section]
\newtheorem{lemma}[theorem]{Lemma}

\newtheorem{corollary}[theorem]{Corollary}
\newtheorem{conjecture}[theorem]{Conjecture}
\newtheorem{question}[theorem]{Question}

\theoremstyle{definition}
\newtheorem{definition}[theorem]{Definition}

\begin{document}

\title[The incompatibility of crossing number and bridge number]{The incompatibility of crossing number and bridge number for knot diagrams}

\author{Ryan Blair}
\address{Department of Mathematics, California State University, Long Beach, 1250 Bellflower Blvd, Long Beach, CA 90840}
\email{ryan.blair@csulb.edu}

\author{Alexandra A. Kjuchukova}
\address{Department of Mathematics, University of Wisconsin-Madison, Van Vleck Hall, 480 Lincoln Drive, Madison, WI 53706}
\email{kjuchukova@math.wisc.edu}

\author{Makoto Ozawa}
\address{Department of Natural Sciences, Faculty of Arts and Sciences, Komazawa University, 1-23-1 Komazawa, Setagaya-ku, Tokyo, 154-8525, Japan}
\email{w3c@komazawa-u.ac.jp}
\thanks{The third author is partially supported by Grant-in-Aid for Scientific Research (C) (No. 26400097 \& 17K05262), The Ministry of Education, Culture, Sports, Science and Technology, Japan}

\subjclass[2010]{Primary 57M25; Secondary 57M27}

\keywords{knot, knot diagram, crossing number, bridge number, perpendicular bridge number, Wirtinger number}

\begin{abstract}
We define and compare several natural ways to compute the bridge number of a knot diagram. We study bridge numbers of crossing number minimizing diagrams, as well as the behavior of diagrammatic bridge numbers under the connected sum operation. For each notion of diagrammatic bridge number considered, we find crossing number minimizing knot diagrams which fail to minimize bridge number. Furthermore, we construct a family of minimal crossing diagrams for which the difference between diagrammatic bridge number and the actual bridge number of the knot grows to infinity.
\end{abstract}

\maketitle


\section{Introduction}

Let $K$ be a knot-type in $\mathbb{R}^3$ and let $\gamma\in K$ be a smooth embedding of the knot-type $K$. Let $p:\mathbb{R}^3\rightarrow \mathbb{R}^2$ be given by $p(x,y,z):=(x,y)$,  $h:\mathbb{R}^3\rightarrow \mathbb{R}$ by $h(x,y,z):=z$ and $h_{||}:\mathbb{R}^3\rightarrow \mathbb{R}$ by $h_{||}(x,y,z)=y$.

Denote by  $\mathcal{C}(K)$ the set of embeddings $\gamma\in K$ which are regular and have minimal crossing number with respect to $p$; namely, these are crossing number minimizing embeddings.
Denote by $\mathcal{B}(K)$ the set of embeddings $\gamma\in K$ which are Morse and have the minimal number of maximal points with respect to $h$; namely, these are bridge number minimizing embeddings.

We have experimentally verified in \cite{Ksite} that $\mathcal{C}(K)\cap \mathcal{B}(K)\neq\emptyset$ for at least 450,000 prime knots of up to 16 crossings. Is this true for all knots? That is, we ask:

\begin{question}
Can we extract the bridge number of a knot $K$ from a minimal crossing diagram of $K$?
\end{question}

In order to study this question, we introduce several diagrammatic notions of bridge number. These are integers associated to knot diagrams with the property that the minimum of the bridge number of $D$ over all diagrams $D$ of a given knot $K$ is equal to the bridge number of $K$. 

A classical notion of diagrammatic bridge number is that of ``overpass'' bridge number. Call an arc in knot diagram a {\it bridge} if it includes at least one overcrossing. The overpass bridge number of the diagram is the number of bridges in the diagram. It is well-known that this definition of diagrammatic bridge number is not well behaved with respect to minimal crossing number diagrams. For instance, the trefoil knot has bridge number two and crossing number three. However, it is a straight-forward exercise to show that every diagram of the trefoil with overpass bridge number two contains at least four crossings. 
More generally, we show that the minimal overpass bridge number and the minimal crossing number are wholly incompatible.

\begin{theorem}\label{thm:overpass}
Let $K$ be a non-trivial knot. $K$ does not admit a diagram which realizes both the minimal overpass bridge number and the minimal crossing number of $K$.
\end{theorem}


Alternatively, the typical diagrammatic depiction of knots obtaining their bridge number suggests defining the ``parallel'' bridge number of a diagram $D$, $b_{||}(D)$, as the minimal number of maxima of $h_{||}|_{\gamma}$ for any Morse embedding $\gamma$ that projects to $D$. We remark that this definition of diagrammatic bridge number is not very effective at calculating bridge number of a knot.  

We focus instead on the perpendicular bridge number of a knot diagram $D$, $b_{\perp}(D)$ (Definition~\ref{perp}), and the Wirtinger number of  $D$, $\omega(D)$ (Definition~\ref{omega}).  The Wirtinger number is defined combinatorially, and the fact that the minimum value of $\omega(D)$ over all diagrams $D$ of a knot $K$ equals the bridge number of $K$ is non-trivial -- see~\cite{BKVV} for a proof. On the plus side, the Wirtinger number is algorthmically computable. This allowed the detection of bridge numbers for nearly half a million knots from minimal-crossing diagrams, and the tabulation of bridge numbers for the majority of these knots for the first time. By contrast, $b_{\perp}(D)$ is defined geometrically, and it follows easily that taking the minimum of $b_{\perp}(D)$ over all diagrams of a knot gives the bridge number, but $b_{\perp}(D)$ can be challenging to compute directly. Additionally, it is a straight forward exercise to show that $b_{\perp}(D)\leq b_{||}(D)$ for all knot diagrams, making $b_{\perp}(D)$ more effective at calculating the bridge number of a knot from one of its diagrams. Our next result relates $\omega(D)$ and $b_{\perp}(D)$.

\begin{theorem}
\label{omega=perp}
For a knot diagram $D$, $\omega(D)=b_{\perp}(D)$.
\end{theorem}

In light of this, we will sometimes use $b(D)$ to denote either of these quantities for a knot diagram, that is, $b(D):=\omega(D)=b_{\perp}(D)$. Throughout, $\beta(K)$ denotes the bridge number of $K$. 
We leverage the above equality to prove the following.

\begin{theorem}
\label{minmal-gap}
 For every positive integer $n$, there exists an alternating knot $K$ and a crossing number minimizing diagram $D$ of $K$, with the property that $b(D) - \beta(K) \geq n$. 
\end{theorem}

This theorem shows that not all minimal diagrams have Wirtinger number equal to the bridge number. However, all of the examples constructed in the proof of Theorem \ref{minmal-gap} have the property that $K$ is a composite knot and that there exists an alternative minimal crossing diagram $D'$ such that $\beta(K)=b_{\perp}(D')$.  In fact, we conjecture that every knot has a minimal diagram realizing bridge number, that is,

\begin{conjecture}
For any knot $K$, $\mathcal{C}(K)\cap\mathcal{B}(K)\ne\emptyset$.
\end{conjecture}

The above conjecture, if true, together with invariance of the Wirtinger number under flypes, which we conjecture holds, would imply:

\begin{conjecture}
For a minimal diagram $D$ of a prime alternating knot $K$, then $b(D) = \beta(K)$.
\end{conjecture}

\section{Preliminaries}

Let $K$, $p$, $h$ and $\gamma$ be as above. Among the many equivalent definitions of bridge number of a knot, we favor the following one.

\begin{definition}
The \emph{bridge number of $K$}, $\beta(K)$, is the minimal number of maxima of $h|_{\gamma}$ over all $\gamma\in K$ such that $h|_{\gamma}$ is Morse.
\end{definition}


The first definition of diagrammatic bridge number we consider is the following. Let $K, \gamma, p, h$ be as above. When $p|_{\gamma}$ is regular and $|p^{-1}(x, y)|\leq 2, \forall(x, y)\in \mathbb{R}^2$, we say $p(\gamma)$ is a {\it knot projection} for $K$. Hence, every knot projection is a finite, $4$-valent graph in the plane. A {\it knot diagram} of $K$ is a knot projection $p(\gamma)$ together with labels that indicate which strand is the over-strand and which is the under-strand at each double point. Given a diagram $D$ and an embedding $\gamma$ of a knot type $K$, we say that $\gamma$ \emph{presents} $D$ if the following hold:

\begin{enumerate}
\item $h|_{\gamma}$ is Morse;

\item $p(\gamma)$ is a knot projection of $K$;

\item $p(\gamma)$ together with crossing labels is equal to $D$.
\end{enumerate}

\begin{definition}\label{perp}
Given a knot diagram $D$ of knot $K$, define the \emph{perpendicular bridge number of $D$}, $b_{\perp}(D)$, to be the minimal number of maxima of $h|_{\gamma}$ over all $\gamma$ presenting $D$.
\end{definition}

For any diagram $D$ of a knot $K$, if $\gamma$ presents $D$, by definition $\gamma\in K$, so $b_{\perp}(D)\geq\beta(K)$. Furthermore, since, after an arbitrarily small perturbation that preserves the number of maxima of $h|_{\gamma}$, any $\gamma\in K$ has a diagram, $\beta(K)$ is realized as $b_{\perp}(D)$ for some diagram $D$ of $K$. Therefore, $b_{\perp}(D)$ has the desired properties of a diagrammatic bridge number. 
Thoerem~\ref{omega=perp} proves that the perpendicular bridge number of a diagram $D$ equals its Wirtinger number, $\omega(D)$. 

The Wirtinger number of a diagram $D$ is calculated algorithmically and is closely related to the problem of finding the minimal number of Wirtinger generators in $D$ which suffice to generate the group of $K$. The Wirtinger number was introduced in \cite{BKVV}, and it follows from the main theorem therein that $\omega(D)$ constitutes a diagrammatic bridge number in the above sense.  We recall the definition here.

Let $D$ be a knot diagram with $n$ crossings and let $v(D)$ be the set of crossings $c_1$, $c_2$,..., $c_n$ in the plane. Since we think of deleting a neighborhood of each understand from a knot projection to form the diagram $D$, then $D$ consists of $n$ disjoint closed arcs in the plane called \emph{strands}. Denote by $s(D)$ the set of strands $s_1$, $s_2$,..., $s_n$ for the diagram $D$. Two strands $s_i$ and $s_j$ of $D$ are {\it adjacent} if $s_i$ and $s_j$ are the under-strands of some crossing in $D$. 
In what follows we assume that $n>1$ so, in particular, no strand is adjacent to itself.


We call $D$ \textit{k-partially colored} if we have specified a subset $A$ of the strands of $D$ and a function $f: A \to \{1, 2, \dots, k\}$.  We refer to this partial coloring by the tuple $(A, f)$. Next, we define an operation that allows us to pass from one partial coloring on a diagram to another.

\begin{definition}
\label{move}
Given $k$-partial colorings $(A_1, f_1)$ and $(A_2, f_2)$ of $D$, we say $(A_2, f_2)$ is the result of a \textit{coloring move } on $(A_1, f_1)$ if the following conditions hold:
\begin{enumerate}
    \item $A_1 \subset A_2$ and $A_2 \setminus A_1 = \{s_j\}$ for some strand $s_j$ in $D$;
    \item $f_2|_{A_1} = f_1$;
    \item $s_j$ is adjacent to $s_i$ at some crossing $c \in v(D)$, and $s_i\in A_1$;
    \item the over-strand $s_k$ at $c$ is an element of $A_1$;
    \item $f_1(s_i) = f_2(s_j)$.
\end{enumerate}
\end{definition}

We denote the above coloring move by $(A_1, f_1)\to(A_2, f_2)$. Two consecutive coloring moves are illustrated in Figure \ref{coloring.fig}, which we borrowed from~\cite{BKVV}. 
%
We say a knot diagram $D$ is \textit{k-colorable} if there exists a $k$-partial coloring $(A_0, f_0) = (\{s_{i_1}, s_{i_2}, \dots, s_{i_k}\}$, $f_0(s_{i_j}) = j)$ and a sequence of coloring moves which result in coloring the entire diagram. 

\begin{definition} \label{omega} Let $D$ be a knot diagram. The smallest integer $k$ such that $D$ 
is $k$-colorable is the {\it Wirtinger number} of $D$, denoted $\omega(D)$. 
\end{definition}
For a proof that the minimum value of $\omega(D)$ over all diagrams $D$ of a knot $K$ equals the bridge number of $K$ we again refer the reader to~\cite{BKVV}. The Wirtinger number is our main tool for computing bridge numbers of minimal diagrams and comparing these to the bridge numbers of the corresponding knots.


\begin{figure}
	\includegraphics[width=4in]{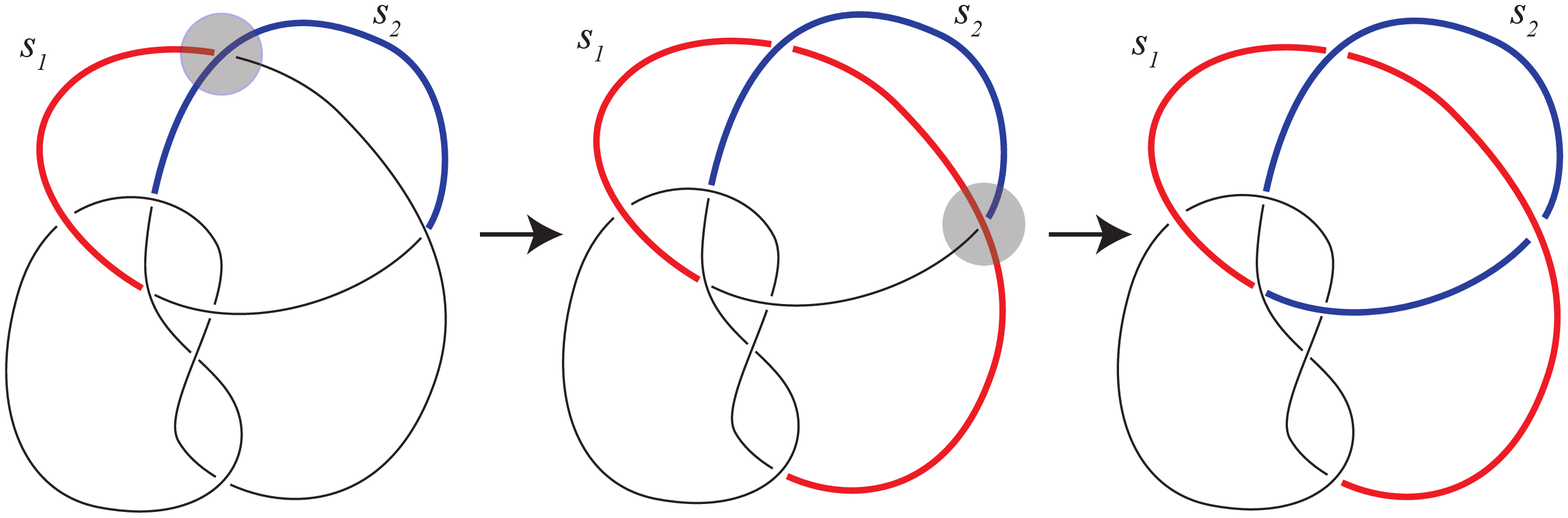}
	\caption{Two coloring moves on the knot $8_{17}$, corresponding to the shaded crossings. The coloring process terminates at this stage. More generally, this diagram is not 2-colorable. $8_{17}$ is a three-bridge knot.}
	\label{coloring.fig}
	\end{figure}
	



In Section \ref{main} we prove that $\omega(D)=b_{\perp}(D)$ for any knot diagram $D$. In Section~\ref{applications} we analyze how Wirtinger number behaves with respect to diagrammatic connected sum and we use these results to show that given any knot $K$ you can always find a diagram $D$ such that the difference between $\omega(D)$ and $\beta(K)$ is arbitrarily large. Finally, we extend these results to show that the difference between $\omega(D)$ and $\beta(K)$ can be arbitrarily large even among crossing number minimizing diagrams if composite knots are allowed. We conclude by exhibiting a crossing number minimizing diagram of a {\it prime} knot which also fails to realize the Wirtinger number.


\section{Main Theorem}\label{main}

Theorem \ref{thm:overpass} follows from the next two lemmas.
We recall from \cite{Schu54} the definition of an overpass (resp. underpass) in a knot diagram.
For a diagram $D=p(\gamma)$ of a knot type $K$, an {\em overpass} (resp. {\em underpass}) is a subarc $p(\alpha)$, where $\alpha$ is a subarc of $\gamma$ and $p|_\alpha$ is an injection, which contains only over-crossings (resp. under-crossings).
Any knot diagram $D$ can be decomposed into an alternating sequence of over- and under-passes $\alpha_1^+, \alpha_1^-,\ldots,\alpha_n^+,\alpha_n^-$. Let us rephrase the definition of overpass bridge number, which we recalled in the introduction, in this language.

\begin{definition} \label{overpass}
The {\em overpass bridge number} of a knot diagram $D$ as the minimal number of $n$ over all alternating sequences of over/underpasses $\alpha_1^+, \alpha_1^-,\ldots,\alpha_n^+,\alpha_n^-$. 
\end{definition}
 
It is well known that the {\em overpass bridge number} of a knot type $K$, namely, the minimal overpass bridge number over all diagrams of $K$, is simply the bridge number of $K$.

\begin{lemma}
If a knot diagram has the minimal overpass bridge number,
then any pair of consecutive over/underpasses intersect in
at least one crossing.
\end{lemma}

\begin{proof}
Suppose that there exists a pair of consecutive
over/underpasses which has no crossing. 
Then we can obtain a knot diagram with a smaller overpass bridge number, by applying a move as shown in Figure~\ref{move1}.
\begin{figure}[htbp]
	\begin{center}
	\includegraphics[trim=0mm 0mm 0mm 0mm, width=.8\linewidth]{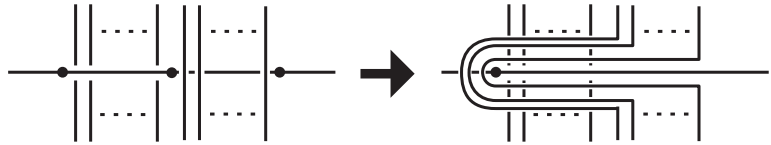}
	\end{center}
	\caption{Reducing the overpass bridge number.}
	\label{move1}
\end{figure}
\end{proof}

\begin{lemma}
Let $D$ be a crossing number minimizing knot diagram. Any pair of consecutive
over/underpasses in $D$ do not intersect at a crossing.
\end{lemma}

\begin{proof}
Consider a knot diagram which contains a pair of consecutive over/underpasses intersecting in
at least one crossings.
Apply a move as shown in Figure \ref{move2}. This results in a knot diagram for the same knot whose crossing number is smaller.
\begin{figure}[htbp]
	\begin{center}
	\includegraphics[trim=0mm 0mm 0mm 0mm, width=.8\linewidth]{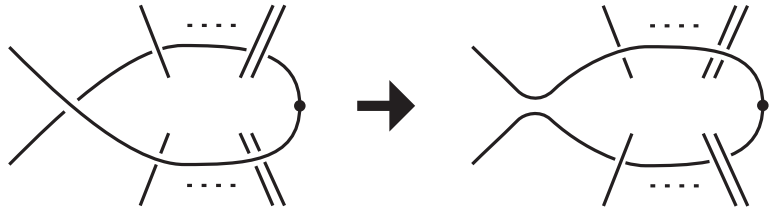}
	\end{center}
	\caption{Reducing the crossing number.}
	\label{move2}
\end{figure}
\end{proof}

Theorem \ref{omega=perp} relates the Wirtinger number of a diagram to its perpendicular bridge number. It is proved in~\cite{BKVV} that $\omega(D)$ can be calculated by computer from Gauss code for $D$, while the more geometric $b_{\perp}(D)$ tends to be rather elusive.


To prove Theorem \ref{omega=perp}, some additional notation must be established. As above, given a diagram $D$, $s(D)$ denotes the set of strands $s_1$, $s_2$,..., $s_n$ and $v(D)$  the set of crossings $c_1$, $c_2$,..., $c_n$. 
Furthermore, if $\gamma$ presents $D$, we denote by $\overline{s}_1$, $\overline{s}_2$,..., $\overline{s}_n$  the subarcs of $p(\gamma)$ obtained by extending each of $s_1$, $s_2$,..., $s_n$ slightly so that the endpoints of the arcs $\overline{s}_1$, $\overline{s}_2$,..., $\overline{s}_n$ are contained in $v(D)$. In particular, if we consider $p(\gamma)$ as a finite, 4-valent graph, then $\overline{s}_i$ is the union of all edges in $p(\gamma)$ whose interiors intersect $s_i$. Additionally, let $\{a_i^+,a_i^-\}:=(p|_{\gamma})^{-1}(v_i)$, where $h(a_i^+)>h(a_i^-)$ for all $v_i\in v(D)$. Finally, observe that $(p|_{\gamma})^{-1}(\overline{s}_i)$ consists of a closed arc of positive length, possibly together with a collection of isolated points in $\gamma$, corresponding to crossings of $D$ where $s_i$ is the over-strand. We denote the closed arc component of $(p|_{\gamma})^{-1}(\overline{s}_i)$ by $\hat{s}_i$.

\begin{lemma}\label{critical}
Let $D$ be a knot diagram and $\gamma$ an embedding that presents $D$ and minimizes $b_{\perp}(D)$. Then, after an isotopy of $\gamma$ which fixes $p(\gamma)$ pointwise, we can assume the following:
for every $s_i\in s(D)$, $\hat{s}_i$ contains at most three critical points of $h|_{\gamma}$; if $\hat{s}_i$ contains exactly three critical points of $h|_{\gamma}$, then two of these critical points are minima and one is a maximum; 
all critical points of $h|_{\gamma}$ corresponding to minima are contained in the set $\{a_1^-, a_2^-,...,a_n^-\}$.
\end{lemma}

\begin{proof}

Assume $\gamma$ is an embedding that presents $D$ and minimizes $b_{\perp}(D)$. Consider a strand $s_i\in s(D)$. Let $U$ denote the portion of the interior of $p^{-1}(\overline{s}_i)$ that lies above $\hat{s}_i$. By definition of strand, $U$ is disjoint from $\gamma$. Hence, there is an ambient isotopy of $\gamma$ supported in an arbitrarily small open neighborhood of $U$ in $\mathbb{R}^3$, after which $\hat{s}_i$ is replaced by an arc with one maximum and at most two minima and $p(\gamma)$ is preserved pointwise.
See Figure~\ref{3crit.fig}.  This is a contradiction to the minimality of $h|_{\gamma}$, unless $\hat{s}_i$ contains at most three critical points of $h|_{\gamma}$. Moroever, this isotopy shows that we can assume that if $\hat{s}_i$ contains exactly three critical points of $h|_{\gamma}$, then two of these critical points are minima and one is a maximum.

As an intermediate next step, we arrange that the critical points of $h|_{\gamma}$ are contained in the interiors of the $\hat{s}_i$. Indeed, since $(p|_{\gamma})^{-1}(v(D))$ is a discrete set in $\gamma$, after an arbitrarily small ambient isotopy of $\gamma$ that fixes $p(\gamma)$ pointwise and preserves that number of critical points, we can assume that no critical point of $h|_{\gamma}$ is contained in $p^{-1}(v(D))$.

\begin{figure}[h]
	\includegraphics[width=1.2in]{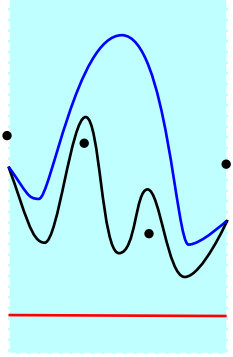}
	\caption{Isotopy decreasing the number of critical points for $\hat{s}_i$, if it contains more than three critical points of $h|_{\gamma}$. The red line is the projection, $s_i$, and the black dots represent the lifts of strands over which $s_i$ passes. }
	\label{3crit.fig}
	\end{figure}

We have already shown that, if $\hat{s}_i$ contains exactly three critical points of $h|_{\gamma}$, then two of these critical points are minima and one is a maximum. To conclude the proof, we need to isotope $\gamma$, preserving $p(\gamma)$ pointwise, so that all critical points of $h|_{\gamma}$ corresponding to minima are contained in the set $\{a_1^-, a_2^-,...,a_n^-\}$. By the above, each arc $\overline{s}_i$ fits into one of three mutually disjoint categories:

\begin{enumerate}
\item Type I: $\hat{s}_i$ contains no minima of $h|_{\gamma}$.

\item Type II: $\hat{s}_i$ contains exactly one minimum and no maximum of $h|_{\gamma}$.

\item Type III: $\hat{s}_i$ contains exactly one minimum and exactly one maximum of $h|_{\gamma}$.

\item Type IV: $\hat{s}_i$ contains exactly two minima and exactly one maximum of $h|_{\gamma}$.

\end{enumerate}

In each of these cases, we let $U$ denote the portion of the interior of $p^{-1}(\overline{s}_i)$ that lies above $\hat{s}_i$. Recall that, by the definition of strand, $U$ is disjoint from $\gamma$.

If $\overline{s}_i$ is an edge of Type II, let $v_j$ and $v_k$ be the endpoints of $\overline{s}_i$ and suppose $h(a_j^-)\leq h(a_k^-)$. Let $U'$ be the portion of $U$ below the plane $\{z=h^{-1}(h(a_j^-))\}$. Then there is an ambient isotopy of $\gamma$ supported in an arbitrarily small open neighborhood of $U'$ in $\mathbb{R}^3$ after which $\hat{s}_i$ is replaced by an arc with no critical points in its interior, $h|_{\gamma}$ has a minimum at $a_j^-$, and $p(\gamma)$ is preserved pointwise. See Figure~\ref{type2.fig}.

\begin{figure}[h!]
\begin{picture}(225,150)
	\put(0,0){\includegraphics[scale=.29]{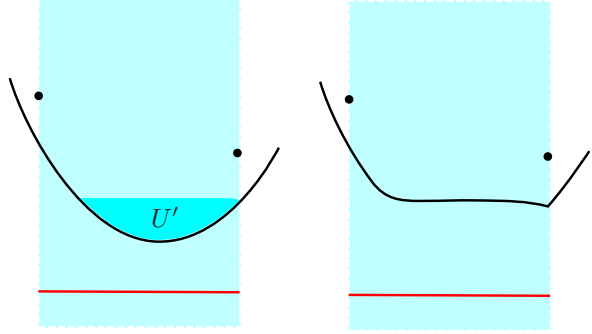}}
	\put(54,40){$U'$}
	\end{picture}
	\caption{Isotopy for a strand of Type II.}
	\label{type2.fig}
	\end{figure}


If $\overline{s}_i$ is an edge of Type III, let $M_i$ be the maximum of $h|_{\gamma}$ in $\hat{s}_i$ and let $m_i$ be the minimum of $h|_{\gamma}$ in $\hat{s}_i$. Let $v_j$ and $v_k$ be the endpoints of $\overline{s}_i$ and suppose $a_j^-$ is closer to $m_i$ than it is to $M_i$, along $\hat{s}_i$. Let $U^*$ be the portion of $U$ below the plane $\{z=h^{-1}(h(a_j^-))\}$ and let $U'$ be the component of $U^*$ whose closure in $U$ contains $m_i$ and $a_j^-$. Note that $M_i$ is not contained in the closure of $U'$ as this would imply that there is an ambient isotopy of $\gamma$ that eliminates a minimum and a maximum while preserving $p(\gamma)$, a contradiction to the minimality assumption for $\gamma$. Subsequently, there is an ambient isotopy of $\gamma$ supported in an arbitrarily small open neighborhood of $U'$ in $\mathbb{R}^3$ after which $\hat{s}_i$ is replaced by an arc with a single maximum critical point in its interior, $h|_{\gamma}$ has a minimum at $a_j^-$, and $p(\gamma)$ is preserved pointwise. See Figure~\ref{type3.fig}.

\begin{figure}[h!]
\begin{picture}(225,150)
	\put(0,0){\includegraphics[scale=.3]{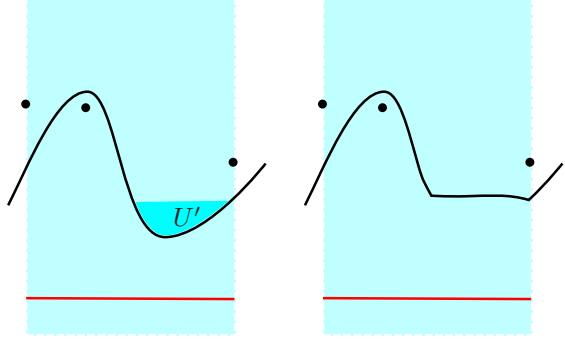}}
	\put(63,42){$U'$}
	\end{picture}
	\caption{Isotopy for a strand of Type III.}
	\label{type3.fig}
	\end{figure}

If $\overline{s}_i$ is an edge of Type IV, let $M_i$ be the maximum of $h|_{\gamma}$ in $\hat{s}_i$ and let $m^1_i$ and $m^2_i$ be the two minima of $h|_{\gamma}$ in $\hat{s}_i$. Let $v_j$ and $v_k$ be the endpoints of $\overline{s}_i$ and suppose $a_j^-$ is closer to $m^1_i$ than it is to $m^2_i$, along $\hat{s}_i$. Let $U^*$ be the portion of $U$ below the plane $\{z=h^{-1}(h(a_j^-))\}$. Note that $h(M_i)>max\{h(a_j^-),h(a_k^-)\}$, since otherwise there would exists an ambient isotopy of $\gamma$ which eliminates a minimum and a maximum while preserving $p(\gamma)$, a contradiction to the minimality assumption for $\gamma$. Hence, the portion of $U$ below the plane $\{z=h^{-1}(h(a_j^-))\}$ contains a disk component $U'$ that is incident to $a_j^-$. Similarly, the portion of $U$ below the plane $\{z=h^{-1}(h(a_k^-))\}$  contains a disk component $U''$ that is incident to $a_k^-$. There is an ambient isotopy of $\gamma$ supported in an arbitrarily small open neighborhood of $U'\cup U''$ in $\mathbb{R}^3$ after which $\hat{s}_i$ is replaced by an arc with a single maximum critical point in its interior, $h|_{\gamma}$ has a minimum at $a_j^-$ and at $a_k^-$, and $p(\gamma)$ is preserved pointwise. See Figure~\ref{type4.fig}.

\begin{figure}[h!]
\begin{picture}(235,150)
	\put(0,0){\includegraphics[scale=.3]{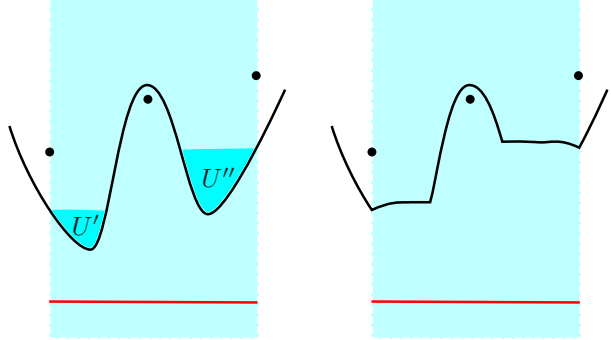}}
	\put(24,40){$U'$}
	\put(73,58){$U''$}
	\end{picture}
	\caption{Isotopy for a strand of Type IV.}
	\label{type4.fig}
	\end{figure}


Since each of the ambient isotopies described above are supported in pairwise disjoint regions in $\mathbb{R}^3$, we can simultaneously apply the above isotopies to produce an ambient isotopy of $\gamma$ that fixes $p(\gamma)$ pointwise, results in an embedding of $\gamma$ that minimizes $b_{\perp}(D)$, and results in all critical points of $h|_{\gamma}$ corresponding to minima being contained in $\{a_1^-, a_2^-,...,a_n^-\}$.
\end{proof}

\begin{proof}  [Proof of Theorem~\ref{omega=perp}]
Given a diagram $D$ of a knot $K$, let $\mu:=\omega(D)$. As in the proof of main result in \cite{BKVV}, we can construct an embedding $\gamma$ that presents $D$ such that $h|_{\gamma}$ has $\mu$ maxima. Hence, $b_{\perp}(D)\leq w(D)$. It remains to show that $b_{\perp}(D)\geq w(D)$.

Let $D$ be a diagram of $K$ and let $\gamma$ be an embedding of $K$ that presents $D$ and realizes $b_{\perp}(D)$. Assume that we have isotoped $\gamma$ so that the conclusions of Lemma~\ref{critical} hold. That is, all critical points corresponding to minima of $h|_{\gamma}$ are contained in $\{a_1^-, a_2^-,...,a_n^-\}$, all critical points corresponding to maxima of  $h|_{\gamma}$ are contained in the interior of the arcs $\hat{s}_i$ for $i\in\{1,...,n\}$, and each arc $\hat{s}_i$ contains at most one critical point corresponding to a maximum of $h|_{\gamma}$. Moreover, after a further isotopy of $\gamma$, fixing $p(\gamma)$ pointwise as always, we can additionally assume that all critical points of $h|_{\gamma}$ and all points in $(p|{\gamma})^{-1}(v(D))$ take distinct values under $h$ and that all critical points corresponding to maxima of $h|_{\gamma}$ lie above all of the points $(p|{\gamma})^{-1}(v(D))$.

Define $g:s(D)\rightarrow \mathbb{R}$ by $g(s_i)=max(\{h(t)|t\in \hat{s}_i\})$. Note that this value is well-defined since $\hat{s}_i$ is compact and that $g$ is one-to-one since all critical points of $h|_{\gamma}$ and all points in $(p|{\gamma})^{-1}(v(D))$ take distinct values under $h$. Order the elements of $s(D)$ in order of decreasing value under the map $g$, and denote the resulting sequence by $s_{i_1},...,s_{i_n}$. Let $b_{\perp}(D)=m$ and observe that the above assumptions guarantee that $s_{i_1},...,s_{i_m}$ are exactly the strands which contain maxima of $h|_{\gamma}$. Set $A_k:=\{s_{i_1},...,s_{i_{m+k}}\}$ for $0\leq k\leq n-m$ and define $f_0: A_0 \rightarrow \{1,...,m\}$ by $f_0(s_{i_j}):=j$ for $1\leq j\leq m$.

Let $M=\{d_1,...,d_m\}\subset \{a_1^-, a_2^-,...,a_n^-\}$ be the collection of critical points corresponding to minima for $h|_{\gamma}$. The closure of each component of $\gamma\setminus M$ is then a union of (consecutive) arcs $\hat{s}_{j}$ and contains a unique critical point corresponding to a maximum of $h|_{\gamma}$. Therefore, each component of $\gamma\setminus M$  contains the interior of a unique arc in the set $\{\hat{s}_{i_1},...,\hat{s}_{i_m}\}$. Extend $f_0$ to a function $f:s(D)\rightarrow \{1,...,m\}$ by assigning $f(s_j):=f(s_{i_k})=k$ where $\hat{s}_{i_k}\in \{\hat{s}_{i_1},...,\hat{s}_{i_m}\}$ is the unique element of this set with the property that  $int(\hat{s}_j)$ and $int(\hat{s}_{i_k})$ are contained in the same component of $\gamma\setminus M$. Define $f_k=f|_{A_k}$ for all $1\leq k\leq n-m$.

To show that $D$ is $m$-meridionally colorable, it remains to be shown that $(A_k, f_k)\to(A_{k+1}, f_{k+1})$ is a valid coloring move for all $0\leq k\leq n-m-1$. Note that $A_{k+1}\setminus A_k=\{s_{i_{m+k+1}}\}$ and let $p:=f(s_{i_{m+k+1}}), p\in \{1,...,m\}$. By the definition of $f$, we have that $\hat{s}_{i_{m+k+1}}$ is in the same component of $\gamma\setminus M$ as $\hat{s}_{i_{p}}$. But $\hat{s}_{i_{m+k+1}}$ does not contain a maximum of $h|_{\gamma}$, since $m+k+1>m$. Therefore, $s_{i_{m+k+1}}$ is adjacent to a strand $s_{i_r}$ such that $f(s_{i_r})=p$ and $g(s_{i_r})>g(s_{i_{m+k+1}})$. Since $g(s_{i_r})>g(s_{i_{m+k+1}})$, we have that $s_{i_r}\in A_k$. Let $s_{i_l}$ be the strand which passes over the crossing $v_r$ at which $s_{i_{m+k+1}}$ and $s_{i_r}$ are adjacent. Since $\hat{s}_{i_{m+k+1}}$ does not contain a maximum of $h|_{\gamma}$ and $g(s_{i_r})>g(s_{i_{m+k+1}})$, then $g(s_{i_{m+k+1}})=h(a_r^-)$. Since $h(a_r^+)>h(a_r^-)$ and $a_r^+\in \hat{s}_{i_l}$, then $g(s_{i_l})>g(s_{i_{m+k+1}})$ and, thus, $s_{i_l}\in A_k$. Hence, conditions (1)-(5) of Definition~\ref{move} are satisfied and $(A_k, f_k)\to(A_{k+1}, f_{k+1})$ is a valid coloring move for all $0\leq k\leq n-m-1$. Therefore, $D$ is $m$-meridionally colorable. This implies that $w(D)\leq b_{\perp}(D)$ and the theorem follows.
\end{proof}

In \cite{BKVV} the authors computed the Wirtinger number for crossing number minimizing diagrams of all knots of eleven or fewer crossings and verified that the Wirtinger number for these diagrams realized bridge number. By the same method, the authors calculated the bridge number of all knots with crossing number 12. Since then, the authors have also calculated the bridge number of more that 450,000 knots with 16 or fewer crossings. Moreover, all of these knots where found to have a minimal crossing diagram $D$ such that $\omega(D)=\beta(K)$. These results motivate the search for the class of knots such that Wirtinger number is realized in a minimal crossing number diagram. We give some negative results in the next section as well as some ensuing conjectures regarding the compatibility of crossing number and Wirtinger number.


\section{Diagrammatic Connected Sum and Examples}\label{applications}

We prove that the Wirtinger number of a knot diagram is super-additive with respect to the connected sum operation on knot diagrams, and we use this fact to construct minimal diagrams whose diagrammatic bridge number is arbitrarily large compared to the bridge number of the knot. 

A diagram $D$ of a knot $K$ is \emph{prime} if every simple closed curve in the plane of projection that is disjoint from the crossings of $D$ and meets the edges of $p(K)$ transversally in two points
 bounds a disk in the plane of projection that is disjoint from the crossings of $D$. Otherwise, we say $D$ is composite.

If $D$ is composite, $\alpha$ is the simple closed curve in the plane of projection that illustrates that $D$ is composite and $s\in \alpha\setminus D$, then the triple $(D,\alpha,s)$ induces a pair of diagrams $D_1$  and $D_2$ by surgering $D$ along the arc of $\alpha\setminus D$ that contains $s$. See Figure~\ref{diagramconnectedsum}. In this case, we write $D=D_1\# D_2$ and say $D$ is the \emph{connected sum} of $D_1$ and $D_2$.  

\begin{figure}[htbp]
	\begin{center}
	\begin{tabular}{cc}
	\includegraphics[trim=0mm 0mm 0mm 0mm, width=.3\linewidth]{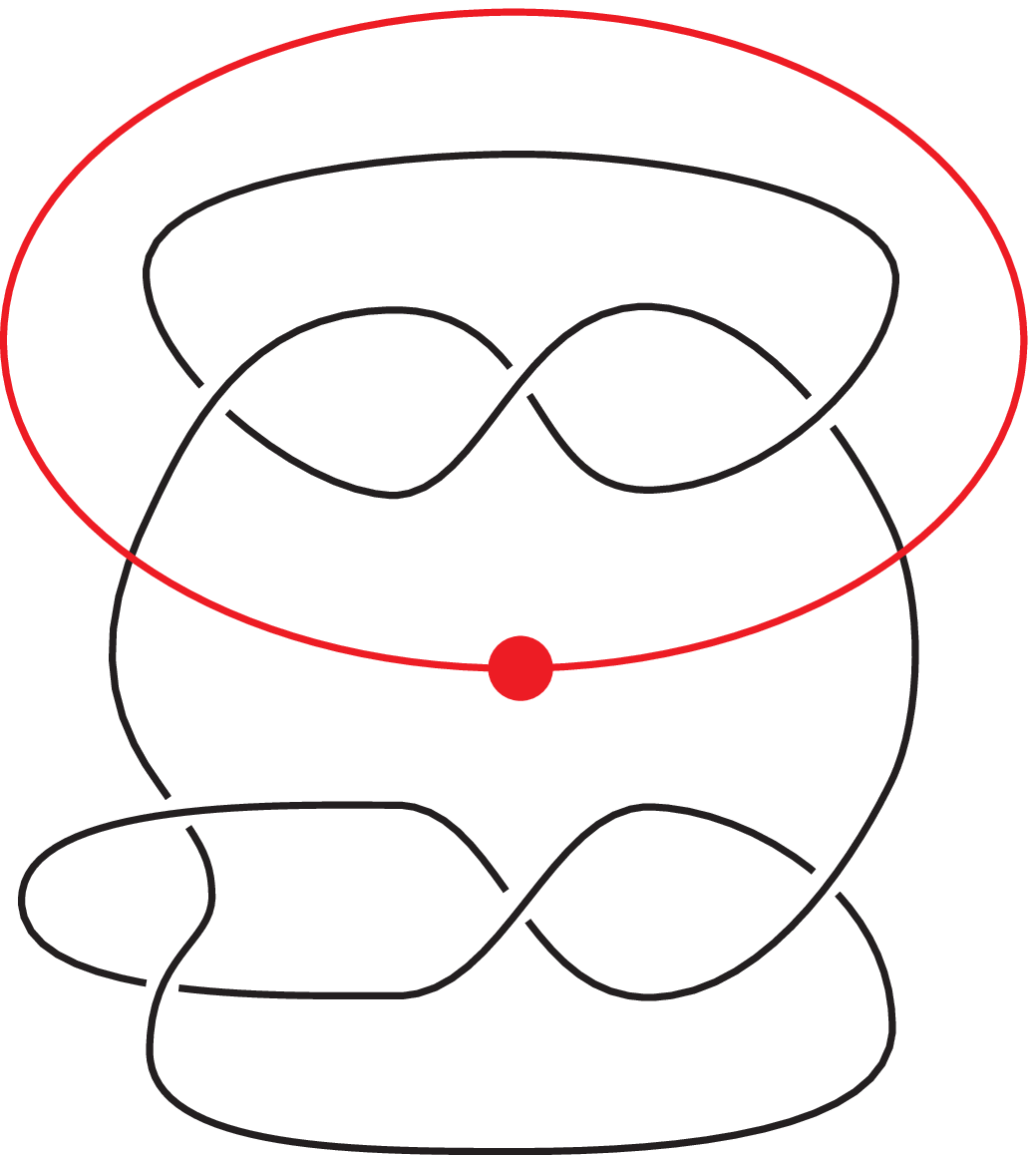}&
	\includegraphics[trim=0mm 0mm 0mm 0mm, width=.3\linewidth]{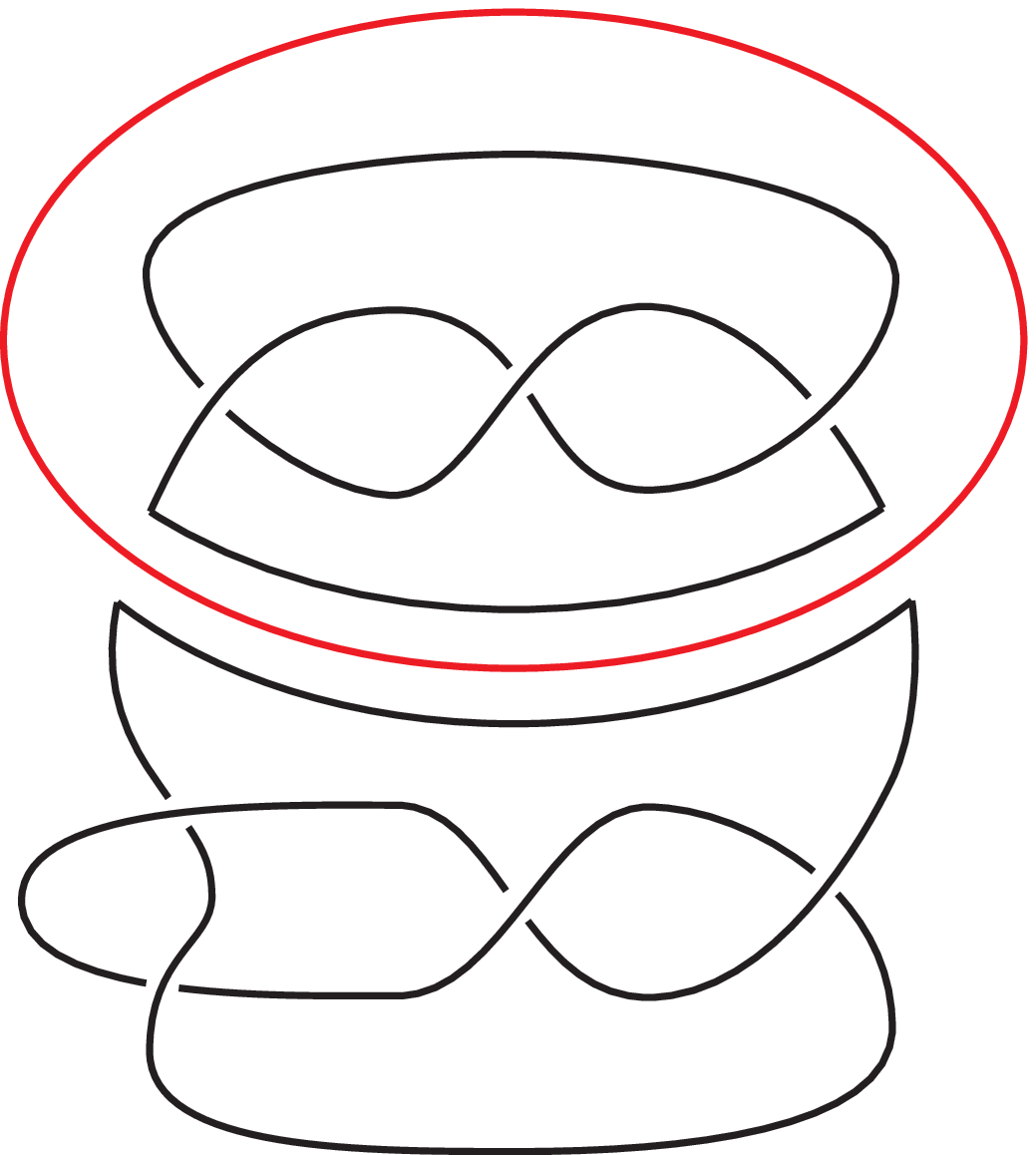}\\
	$(D,\alpha,s)$ & $D_1 \cup D_2$
	\end{tabular}
	\end{center}
	\caption{An example of the triple $(D,\alpha,s)$ inducing a pair of diagrams $D_1$  and $D_2$. The red dot is $s$ and specifies the arc along which the diagram is surgered.}
	\label{diagramconnectedsum}
\end{figure}


\begin{theorem}
\label{add}
Given a composite diagram $D:=D_1\# D_2$, the inequality $b_{\perp}(D)\geq b_{\perp}(D_1)+b_{\perp}(D_2)-1$ holds.
\end{theorem}

\begin{proof}
Let $(D,\alpha,s)$ be the triple that gives rise to the decomposition $D=D_1\# D_2$.
 Let $\gamma\in K$ such that $\gamma$ presents $D$ and $h|_{\gamma}$ has $b_{\perp}(D)$ maxima. 
After a small perturbation of $\alpha$ that is transverse to $p(K)$, we can assume that 
$A=p^{-1}(\alpha)$ is disjoint from the critical points of $h|_{\gamma}$ and $A\cap \gamma =\{x,y\}$
 such that $h(x)<h(y)$. Let $\hat{\alpha}$ be a monotone increasing (wrt the $z$ coordinate) arc embedded in $A$ connecting $x$
 to $y$ that is mapped by $p$ homeomorphically onto a sub arc of $\alpha$. Surgering $\gamma$ along
 $\hat{\alpha}$ with framing parallel to the plane of projection results in two knot embeddings $\gamma_1$ and $\gamma_2$ such that $\gamma_1$
 presents $D_1$ and $\gamma_2$ presents $D_2$.  Independent of the sign of the derivative of 
$h|_{\gamma}$ at $x$ and $y$, $h|_{\gamma_1\cup \gamma_2}$ has exactly one more maxima than 
$h|_{\gamma}$. See Figure \ref{4_1}.

\begin{figure}[htbp]
	\begin{center}
	\includegraphics[trim=0mm 0mm 0mm 0mm, width=.6\linewidth]{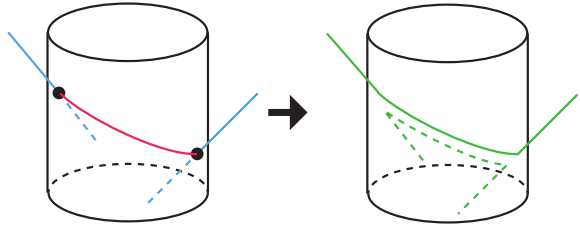}
	\end{center}
	\caption{Surgering $\gamma$ along $\hat{\alpha}$}
	\label{4_1}
\end{figure}

Since the number of maxima of $h|_{\gamma_1\cup \gamma_2}$ 
is bounded below by $b_{\perp}(D_1)+b_{\perp}(D_2)$, then 
$b_{\perp}(D)\geq b_{\perp}(D_1)+b_{\perp}(D_2)-1$.
\end{proof}

\begin{corollary} Given a composite diagram $D:=D_1\# D_2$, the inequality $\omega(D)\geq \omega(D_1)+ \omega(D_2)-1$ holds.
\end{corollary}

\begin{proof} By Theorem~\ref{omega=perp}, for any knot diagram $D$, $\omega(D) = b_{\perp}(D)$.
\end{proof}

\begin{corollary} 
\label{unknot-gap}
Given any integer $n$, there exists a diagram $D_n$ of the unknot such that $b_{\perp}(D)=\omega(D) > n$. 
\end{corollary}

\begin{proof} Let $D$ denote the Thistlethwaite diagram of the unknot, pictured in Figure~\ref{Thistlethwaite}. By direct computation, we find that $\omega(D) =3$. (This can be done by hand. We show $\omega(D) >2$ by verifying that, if we begin by coloring the over-strand and  one under-strand at any crossing, the coloring process terminates before all strands are colored, regardless of the choice of crossing. By contrast, it is not hard to find three strands which, when colored, allow us to extend the coloring to the entire diagram, by iterating the coloring move.) Let $D_0=D_1=D$ and let $D_n$ be a connected sum of $n$ copies of $D$. By repeated application of Theorem~\ref{add}, we have $\omega(D_n) = b_{\perp}(D_n) \geq 3 + 2(n-1) = 2n+1 > n$.
\end{proof}

\begin{figure}[htbp]
	\begin{center}
	\includegraphics[trim=0mm 0mm 0mm 0mm, width=.4\linewidth]{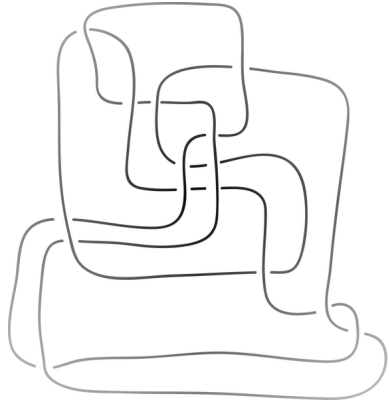}
	\end{center}
	\caption{The Thistlethwaite unknot, whose Wirtinger number is 3.}
	\label{Thistlethwaite}
\end{figure}

This allows us to construct knot diagrams for which the gap between the Wirtinger number and the bridge number of the knot is arbitrarily large.

\begin{corollary}
\label{anyknot-gap}
 For every positive integer $m$ and every knot $K$, there exists a diagram $D$ of $K$ such that $b_{\perp}(D)=\omega(D)> \beta(K) + m$. 
\end{corollary}

\begin{proof} Let $D_1$ be any diagram of $K$, and $D_{n}$ a diagram of the unknot satisfying the conclusion of Corollary~\ref{unknot-gap} for $n=m+1$.  Let  $D:=D_1\# D_n$. By Theorem~\ref{add}, $D$ is a diagram of $K$ such that 
$$\omega(D) \geq \omega(D_1) + \omega(D_n) - 1\geq \omega(K) + \omega(D_n) -1 > \omega(K) + n-1$$
$$ = \omega(K) + m=\beta(K)+m.$$
\end{proof}

Note, however, that the diagram constructed in the proof of Corollary~\ref{anyknot-gap} has a very large number of crossings, exceeding the crossing number of $K$ by at least $15n$. A more subtle question is how big the gap between $\omega(D)$ and $\omega(K)$ can be for $D$ a {\it minimal} diagram of the knot $K$. We demonstrate that this gap can be arbitrarily large as well.


\begin{proof} [Proof of Theorem~\ref{minmal-gap}]
In Figure \ref{example} we give an example of a knot $K$ with reduced alternating diagram $D$ such that $\omega(K)=\beta(K)=5$ and $\omega(D)=6$. The fact that $\omega(D)=6$ was established by computer computation using the algorithm introduced in \cite{BKVV}. By Schubert's equality for bridge number, $\omega(\#_{i=1}^{n} K)= 5n-(n-1)$. Similarly, by Theorem \ref{add}, $\omega(\#_{i=1}^{n} D)\geq 6n-(n-1)$. Hence, $\omega(\#_{i=1}^{n} D)-\omega(\#_{i=1}^{n} K)\geq n$.
\end{proof}

\begin{figure}[htbp]
	\begin{center}
	\includegraphics[trim=0mm 0mm 0mm 0mm, width=.28\linewidth]{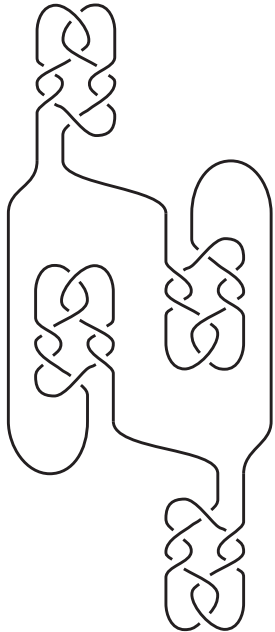}
	\end{center}
	\caption{A minimal crossing diagram $D$ of a composite knot $K$ such that $\omega(D)=6$ and $\beta(K)=5$. Additionally, $b_{||}(D)=7$.}
	\label{example}
\end{figure}

Having seen that the gap between the bridge number of a knot $K$ and diagrammatic bridge number in a minimal diagram of $K$ can be arbitrarily large for composite knots, we naturally ask whether minimal diagrams of {\it prime} knots always realize bridge number. We answer this question in the negative.

\begin{theorem} There exists a prime knot $K$ with  $\mathcal{C}(K)\setminus\mathcal{B}(K)\ne\emptyset$.
\end{theorem}

\begin{proof} Let $K_0$ be the knot in Figure~\ref{example2}, and denote the diagram depicted by $D_0$. Using the algorithm described in~\cite{BKVV}, we calculated the Wirtinger number of this diagram and found that $\omega(D_0)=6$. Moreover, it is straight-forward to verify that the diagram $D_0$ is adequate and thus crossing number minimizing by \cite{Th88}. Due to a somewhat lengthy, but standard argument of analyzing the possible intersections between an essential meridional annulus and the various essential surfaces in the exterior $K_0$, one can show that $K_0$ is prime, see for example Lemma 6.0.29 of \cite{BThesis}. It is similarly straight-forward to show that $K_0$ is an index two cable of the trefoil knot. Hence, by \cite{Schu54}, $\beta(K_0)\geq 4$. However, Figure \ref{example3}, illustrates that $\beta(K_0)\leq4$.
\end{proof}



We believe the method used to construct the knot $K_0$ can be generalized to find other minimal diagrams $D_i$ of prime non-alternating knots $K_i$, such that $b(D_i)>\beta(K_i)$ and the gap between $b(D_i)$ and $\beta(K_i)$ grows with the crossing number of $D_i$.  It is interesting to note that there exists a minimal crossing diagram for $K_0$, given in Figure \ref{example3}, that realizes bridge number.

\begin{figure}[htbp]
	\begin{center}
	\includegraphics[trim=0mm 0mm 0mm 0mm, width=.38\linewidth]{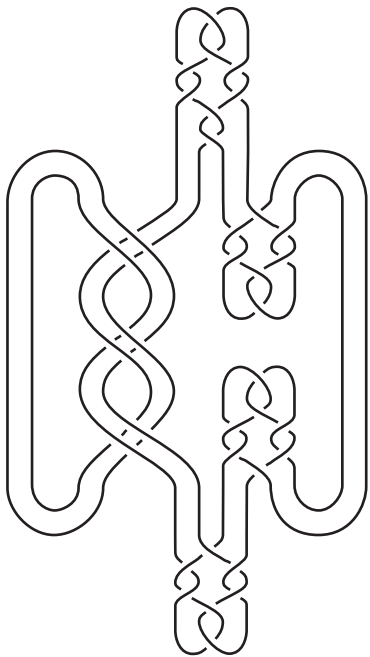}
	\end{center}
	\caption{A minimal crossing diagram $D$ of a prime knot $K$ such that $\omega(D)=6$ and $\beta(K)=4$.}
	\label{example2}
\end{figure}

\begin{figure}[htbp]
	\begin{center}
	\includegraphics[trim=0mm 0mm 0mm 0mm, width=.25\linewidth]{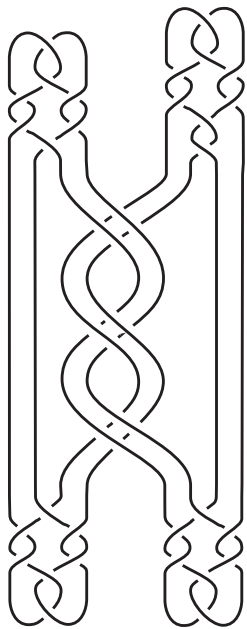}
	\end{center}
	\caption{A minimal crossing diagram $D$ of a prime knot $K$ such that $w(D)=\beta(K)=4$.}
	\label{example3}
\end{figure}






\end{document}